\documentclass[10pt]{article}
\usepackage{amsmath}
\usepackage{amsthm}
\usepackage{enumitem}
\usepackage{amsfonts}
\usepackage{dsfont}
\usepackage{amssymb}
\usepackage{latexsym}
\usepackage{tensor}
\usepackage{verbatim}
\usepackage{graphicx}
\usepackage{xfrac}
\usepackage{tikz}
\usetikzlibrary{cd}
\graphicspath{ {images/} }
\usepackage{pict2e}
\usepackage[font={small}]{caption}



\makeatletter 
\@addtoreset{equation}{section}
\makeatother

\theoremstyle{plain}
\newtheorem{theorem}[equation]{Theorem}
\newtheorem{corollary}[equation]{Corollary}
\newtheorem*{unnum_corollary}{Corollary}
\newtheorem{lemma}[equation]{Lemma}
\newtheorem{proposition}[equation]{Proposition}

\theoremstyle{definition}
\newtheorem{definition}[equation]{Definition}

\newtheorem{example}[equation]{Example}

\newcommand{\fS}{\mathfrak{S}}

\newcommand{\sB}{\mathcal{B}}

\newcommand{\ID}{\mathbb{D}}

\newcommand{\IN}{\mathbb{N}}

\newcommand{\IZ}{\mathbb{Z}}

\def\d/{/\mspace{-6.0mu}/}

\tikzset{cross/.style={cross out, draw=black, fill=none, minimum size=2*(#1-\pgflinewidth), inner sep=0pt, outer sep=0pt}, cross/.default={2pt}}

\begin{document}

\title{Characterizing Rothe Diagrams}
\author{Ben Gillen and Jonathan Michala}

\date{March 20, 2023}

\maketitle

\begin{abstract}
    Rothe diagrams are diagrams which track inversions of a permutation.
    We define six main properties that Rothe diagrams fulfill: the southwest, dot, popping, numbering, step-out avoiding, and empty cell gap rules. 
    We prove that -- given an arbitrary bubble diagram -- four different subsets of these properties provide sufficient criteria for the diagram to be a Rothe diagram. 
    We also prove that when a set of ordered, freely floating, non-empty columns satisfy the numbering and step-out avoiding rules, then they can be arranged into a Rothe diagram.
\end{abstract}

\section{Introduction}
Consider a permutation $w$ in $\fS_\infty$, the group of permutations on $\IZ^+$ with finitely many unfixed points.
For non-identity permutations, write a truncated $w = w_1w_2...w_n$, where $w_n$ is the last unfixed point.
The permutation correspondingly sends $i$ to $w_i$.
An \textbf{\textit{inversion}} of $w$ is a pair of indices $i < j$ such that $w_i > w_j$.
For example, if $w = 152869347$, then $w_4 = 8$ and $w_7 = 3$ is an inversion.
One way to visualize inversions is with the \textbf{\textit{Rothe diagram}} of $w$, defined to be the following subset of cells in the first quadrant:
\begin{equation}
\ID(w) = \{(i,w_j) \colon i < j, w_i > w_j\} \subset \IZ^+\! \times \IZ^+.
\end{equation}

Graphically, we represent cells in $\ID(w)$ by bubbles and use the French convention that the $i$ coordinates are written on the vertical axis and the $w_j$ coordinates are written on the horizontal axis.
Figure \ref{fig:rothe} shows the Rothe diagram $\ID(152869347)$.
There is a bubble at $(4,3)$ for example, because $i = 4 < 7 = j$ and $w_i = 8 > 3 = w_j$.

\begin{figure}
\begin{center}
\begin{tikzpicture}[scale=0.3]
\def\rows{8} %number of rows
\def\cols{7} %number of columns
  \draw[-](0,0) -- (\rows,0); %axes
  \draw[-](0,0) -- (0,\cols); 
  \foreach \y in {1,2,...,\cols} %labels
    \draw (-0.5,\y-0.5) node{\y}; 
  \draw (-2,\cols / 2) node{$i$};
  \foreach \x in {1,2,...,\rows}
    \draw (\x-0.5,- 0.5) node{\x};
  \draw (\rows / 2,-2) node{$w_j$}; 
  \foreach \x/\y in {2/2,3/2,4/2,3/4,4/4,6/4,7/4,3/5,4/5,3/6,4/6,7/6} %positions
    \draw (\x-0.5,\y-0.5) circle(0.5);
\end{tikzpicture}
\hspace{2cm}
\begin{tikzpicture}[scale=0.3]
\def\rows{8} %number of rows
\def\cols{7} %number of columns
  \draw[-](0,0) -- (\rows,0); %axes
  \draw[-](0,0) -- (0,\cols); 
  \foreach \y in {1,2,...,\cols} %labels
    \draw (-0.5,\y-0.5) node{\y}; 
  \draw (-2,\cols / 2) node{$i$};
  \foreach \x in {1,2,...,\rows}
    \draw (\x-0.5,- 0.5) node{\x};
  \draw (\rows / 2,-2) node{$w_j$}; 
  \foreach \x/\y in {1/1,2/3,3/7,5/2,6/5,8/4} { %dot positions
    \draw[fill=black] (\x-0.5,\y-0.5) circle(0.2);
    \draw[-latex](\x-0.5,\y-0.5) -- (\x-0.5,7.7);
    \draw[-latex](\x-0.5,\y-0.5) -- (8.7,\y-0.5); }
  \foreach \x/\y in {2/2,3/2,4/2,3/4,4/4,6/4,7/4,3/5,4/5,3/6,4/6,7/6} %positions
    \draw (\x-0.5,\y-0.5) circle(0.5);
\end{tikzpicture}
\caption{\label{fig:rothe} The Rothe diagram $\ID(152869347)$ and the placement of its death rays}
\end{center}
\end{figure}
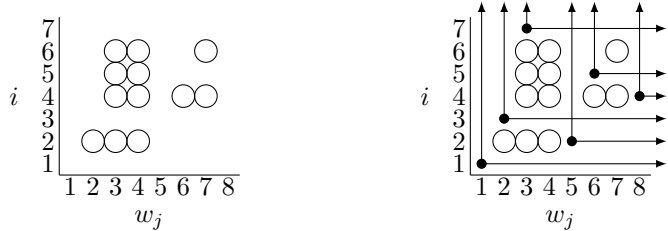

\begin{sloppypar}
There is much research in combinatorics which utilize Rothe diagrams, including \cite{anderson2018diagrams}\cite{bergeron92combinatorial}\cite{fan2019standard}\cite{fan2021set}\cite{fan2022upper}\cite{lambert2018theta}\cite{liu2022schubert}\cite{pawlowski2018cohomology}\cite{pechenik2022regularity}\cite{reiner1998percentage}.
Here we survey some particular uses of Rothe diagrams as well.
Edelman and Greene extend Stanley's results in \cite{stanley1984number} by elucidating the structure of a sum (which depends on Rothe diagrams) that provides the number of reduced decompositions of the most transposed word in the symmetric group.
They give a combinatorial interpretation of the sum's coefficient when the Rothe diagram's column sums are in nondecreasing order or the conjugate of the Rothe diagram's row sums are in nondecreasing order.
They also prove a conjecture from the same paper, narrowing the coefficients of the sum in general from integers to nonnegative integers \cite{edelman1987balanced}.
Linusson and Potka reinterpret the tableaux generated by the Edelman-Greene correspondence in \cite{edelman1987balanced} by using the shape of the bottom-left corner of the Rothe diagram \cite{linusson2018new}.
\end{sloppypar}

Hamaker, Marberg, and Pawlowski define an involution analogue of Rothe diagrams which is a key tool in understanding their corresponding analogue of Stanley symmetric functions.
They show that the involution Stanley symmetric function for the longest element of a finite symmetric group is a product of staircase-shaped Schur functions.
This implies that the number of involution words for the longest element of a finite symmetric group is equal to the dimension of a certain irreducible representation of a Weyl group of type B \cite{hamaker2018involution}.

Here are various other results concerning Rothe diagrams.
In \cite{assaf2020bijective}, Assaf gives a bijective proof that Schubert Polynomials are equivalent to Kohnert Polynomials, which are obtained from Rothe diagrams via Kohnert moves \cite{kohnert1991weintrauben}.
Reiner and Shimozono use Rothe diagrams to characterize a map (called plactification) from reduced words to words;
plactification has applications to the enumeration of reduced words, Schubert Polynomials, and Specht Modules \cite{reiner1995plactification}.
If the Rothe diagram of a permutation is a skew shape, then the corresponding skew Schur function is equal to the Stanley symmetric function of the permutation \cite{billey1993some}.
A permutation is vexillary, or 2143-avoiding, if and only if its Rothe diagram is equivalent to the diagram of a partition \cite{manivel2001symmetric}, if and only if the Stanley symmetric function of the permutation is a Schur function \cite{macdonald1991Schubert}\cite{stanley1984number}.

In this paper, we fully characterize Rothe diagrams with insightful properties in the space of arbitrary diagrams.
In the future, we wish to explore potential connections this characterization may have with the various research uses above.

A well-known initial characterization is the \textbf{Lehmer code} of a permutation $L(w)$, which is a sequence that counts the number of inversions at every index \cite{lehmer1960teaching}.
That is
$$L(w) = (a_1,a_2,...) \in \IN^\infty \text{ such that } a_i = \#\{j > i \colon w_i > w_j\}.$$
The Lehmer code counts the number of bubbles in each row of a Rothe diagram, and it defines a bijection between permutations in $\fS_\infty$ and sequences in $\IN^\infty$ with finitely many nonzero values.
Thus, for such a sequence $a_1,a_2,...$, there is only one Rothe diagram that has $a_i$ bubbles in the $i$th row for all $i$.
For example, the first diagram $\ID(231)$ in Figure \ref{fig:non_rothe_examples} is the only Rothe diagram with one bubble in the first row and one in the second row. 
The second and third diagrams, therefore, cannot be Rothe diagrams, nor can any other diagram with one bubble in the first row and one in the second.
We will use the Rothe diagram in Figure \ref{fig:rothe} and these two non-Rothe diagrams, $D_1$ and $D_2$, as our toy examples throughout the paper.

\begin{figure}
\begin{center}
\begin{tikzpicture}[scale=0.3]
\def\rows{3} %number of rows
\def\cols{3} %number of columns
  \draw[-](0,0) -- (\rows,0); %axes
  \draw[-](0,0) -- (0,\cols); 
  \foreach \x/\y in {1/1,1/2} %positions
    \draw (\x-0.5,\y-0.5) circle(0.5);
\end{tikzpicture}
\hspace{1.4cm}
\begin{tikzpicture}[scale=0.3]
\def\rows{3} %number of rows
\def\cols{3} %number of columns
  \draw[-](0,0) -- (\rows,0); %axes
  \draw[-](0,0) -- (0,\cols); 
  \foreach \x/\y in {1/1,2/2} %positions
    \draw (\x-0.5,\y-0.5) circle(0.5);
\end{tikzpicture}
\hspace{1.4cm}
\begin{tikzpicture}[scale=0.3]
\def\rows{3} %number of rows
\def\cols{3} %number of columns
  \draw[-](0,0) -- (\rows,0); %axes
  \draw[-](0,0) -- (0,\cols); 
  \foreach \x/\y in {2/1,1/2} %positions
    \draw (\x-0.5,\y-0.5) circle(0.5);
\end{tikzpicture}
    \caption{$\ID(231)$ and two non-Rothe diagrams, $D_1$ and $D_2$ (respectively)}
    \label{fig:non_rothe_examples}
\end{center}
\end{figure}
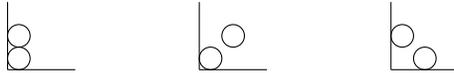

Hence, the Lehmer code can tell us which diagrams are Rothe diagrams.
Given any finite bubble diagram $D \subset \IZ^+ \!\times \IZ^+$, we count the number of bubbles in each row.
We then find the Rothe diagram corresponding to this sequence using Lehmer's bijection.
If the two diagrams match, then $D$ is a Rothe diagram.
If not, then $D$ is not a Rothe diagram.
However, this characterization doesn't provide much intuition about the general properties of a diagram that allow it to be a Rothe diagram.
In this paper, we find such properties, giving the following full characterizations of Rothe diagrams.

\begin{theorem}
\label{thm:main}
Given a diagram D, the following are equivalent:
\begin{enumerate}[label=(\roman*)]
    \item $D$ is a Rothe diagram.
    \item $D$ satisfies the vertical popping and emtpy cell gap rules.
    \item $D$ satisfies the numbering and dot rules.
    \item $D$ satisfies the dot and southwest rules.
    \item $D$ satisfies the numbering rule and is step-out avoiding.
\end{enumerate}
\end{theorem}

In a more general setting, instead of being given a diagram, we are only given a collection of columns ordered from left to right.
We ask if we can construct a Rothe diagram with these columns in this order.
Based on our characterization of Rothe diagrams, we get the following answer with appropriate reformulations of the numbering and step-out avoiding rules.

\begin{corollary}
\label{cor:free_columns}
An ordered collection of free columns may be placed uniquely into a Rothe diagram if and only if the collection satisfies the numbering rule and is step-out avoiding.
\end{corollary}

In the next section we define the properties for bubble diagrams mentioned above and prove that they are necessary in Rothe diagrams.
Then, in the following section, we show how these properties combine to sufficiently characterize Rothe diagrams.
Finally, in the last section, we consider ordered sets of nonempty columns.
We use our established properties to determine when a Rothe diagram can be built out of a given set of these columns.

\section{Properties of Rothe diagrams}

In this section we define the properties mentioned in Theorem \ref{thm:main} and prove that Rothe diagrams satisfy them.
As we talk about these properties, it will be useful to have the following equivalent construction of the Rothe diagram $\ID(w)$.
Add dots to $\IZ^+\!\times \IZ^+$ in the cells $(i,w_i)$ for all $i$.
At each dot, draw two infinite rays coming out of it, one pointing up and one pointing right.
We call these rays \textbf{\textit{death rays}} and the dot at $(i,w_i)$ a \textbf{\textit{death ray origin}}.
Then, the bubbles in $\ID(w)$ should be placed in the cells not touched by a death ray.
Indeed, a bubble is placed in $(i,w_j)$ if and only if $i < j$ and $w_i > w_j$, which is to say $(i,w_j)$ is below $(j,w_j)$ and to the left of $(i,w_i)$.
These are the locations of the death ray origins in the $w_j$th column and $i$th row respectively, and being below and to the left of death ray origins is equivalent to not being touched by a death ray.
See the second diagram in Figure \ref{fig:rothe} for this alternate visualisation.
This analysis can be written in a very useful statement:
\begin{lemma}
\label{lem:sw_dot}
In a Rothe diagram, a cell is filled with a bubble if and only if it is to the left of a death ray origin and below a death ray origin.
\end{lemma}

\subsection{Southwest Rule}

\begin{definition}
As defined in \cite{reiner1995plactification}, a diagram $D \subset \IZ^+ \!\times \IZ^+$ is \textbf{southwest} if $(i,j) \in D$ and $(i',j') \in D$ imply $(\min(i,i'),\min(j,j')) \in D$. 
\end{definition}

\begin{proposition}
Rothe diagrams are southwest.
\end{proposition}

\begin{proof}
Consider two bubbles $(i,w_j)$ and $(i',w_{j'})$ in a Rothe diagram $\ID(w)$.
Using death rays and Lemma \ref{lem:sw_dot}, the cell $A =(\min(i,i'),\min(w_j,w_{j'}))$ is weakly to the left of $(i',w_{j'})$, which is to the left of a death ray origin.
And $A$ is weakly below $(i,w_j)$ which is below a death ray origin.
Therefore, there is a bubble in cell $A$.
\end{proof}

\subsection{Dot and Popping Rules}

Now, we define two notions of placing dots on any diagram and see that on Rothe diagrams, these agree with each other and coincide with the placement of death ray origins.

\begin{definition}
Consider a diagram $D \subset \IZ^+ \!\times \IZ^+$.
Inductively define $c_i = \min\{j \in \IZ^+ \colon j > j' \;\forall\; (i,j') \in D, j \neq c_{i'} \;\forall\; i' < i\}$.
Then define $R(D) = \{(i,c_i)\}_{i=1}^\infty$ to be the set of cells in row $i$ and column $c_i$, which we will call the \textbf{\textit{row dots}} of $D$.
Similarly, inductively define $r_j = \min\{i \in \IZ^+ \colon i > i' \;\forall\; (i',j) \in D, i \neq r_{j'} \;\forall\; j' < j\}$.
Then define $C(D) = \{(r_j,j)\}_{j=1}^\infty$ to be the set of cells in row $r_j$ and column $j$, which we will call the \textbf{\textit{column dots}} of $D$.
We say that a diagram satisfies the \textbf{\textit{dot rule}} and is a \textbf{\textit{dotted diagram}} if $R(D) = C(D)$.
\end{definition}

To place the row dots on a diagram, starting in row 1 and going up, we place a dot in the first cell to the right of all the bubbles in the row, such that we don't place it in the same column as an existing dot.
To place the column dots, moving left to right, we place a dot in the first cell above all of the bubbles in a column, such that we don't place it in the same row as an existing dot.

\begin{example}
In figure \ref{fig:main_dot}, we place row and column dots on $\ID(152869347)$, $D_1$, and $D_2$.
The row and column dots agree in $\ID(152869347)$ and $D_2$, so they satisfy the dot rule.
However $D_1$ does not satisfy the dot rule, because $R(D_1) \neq C(D_1)$.
\end{example}

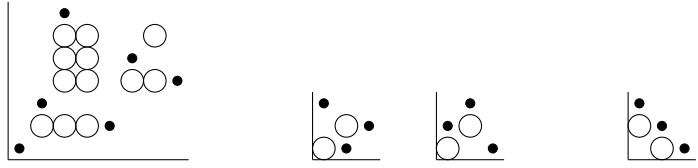
\begin{figure}
\begin{center}
\begin{tikzpicture}[scale=0.3]
\def\rows{8} %number of rows
\def\cols{7} %number of columns
  \draw[-](0,0) -- (\rows,0); %axes
  \draw[-](0,0) -- (0,\cols); 
  \foreach \x/\y in {1/1,2/3,3/7,5/2,6/5,8/4} {%dot positions
    \draw[fill=black] (\x-0.5,\y-0.5) circle(0.2); }
  \foreach \x/\y in {2/2,3/2,4/2,3/4,4/4,6/4,7/4,3/5,4/5,3/6,4/6,7/6} %positions
    \draw (\x-0.5,\y-0.5) circle(0.5);
\end{tikzpicture}
\hspace{1.4cm}
\begin{tikzpicture}[scale=0.3]
\def\rows{3}
\def\cols{3}
  \draw[-](0,0) -- (\rows,0);
  \draw[-](0,0) -- (0,\cols);
  \foreach \x/\y in {1/1,2/2} %circles
    \draw (\x-0.5,\y-0.5) circle(0.5);
  \foreach \x/\y in {2/1,3/2,1/3}{ %dot-lower empty
    \draw[fill=black] (\x-0.5,\y-0.5) circle(0.2); }
\end{tikzpicture}
\hspace{0.5cm}
\begin{tikzpicture}[scale=0.3]
\def\rows{3}
\def\cols{3}
  \draw[-](0,0) -- (\rows,0);
  \draw[-](0,0) -- (0,\cols);
  \foreach \x/\y in {1/1,2/2} %circles
    \draw (\x-0.5,\y-0.5) circle(0.5);
  \foreach \x/\y in {1/2,2/3,3/1}{ %dot-upper full
    \draw[fill=black] (\x-0.5,\y-0.5) circle(0.2); }
\end{tikzpicture}
\hspace{1.4cm}
\begin{tikzpicture}[scale=0.3]
\def\rows{3} %number of rows
\def\cols{3} %number of columns
  \draw[-](0,0) -- (\rows,0); %axes
  \draw[-](0,0) -- (0,\cols); 
  \foreach \x/\y in {1/3,2/2,3/1} {%dot positions
    \draw[fill=black] (\x-0.5,\y-0.5) circle(0.2); }
  \foreach \x/\y in {2/1,1/2} %positions
    \draw (\x-0.5,\y-0.5) circle(0.5);
\end{tikzpicture}
  \caption{\label{fig:main_dot}Row and column dot placement in diagrams $\ID(152869347)$, $D_1$, and $D_2$.}
\end{center}
\end{figure}

We now show that on Rothe diagrams, the row and column dots agree and are the death ray origins.

\begin{proposition}\label{pd=>dot}
For a Rothe diagram $\ID(w)$, we have $R(\ID(w)) = C(\ID(w)) = \{(i,w_i) \in \IZ^+\!\times \IZ^+\}$.
\end{proposition}

\begin{proof}
First, we find $R(\ID(w))$.
In the first row, we have a death ray origin at $(1,w_1)$ meaning every cell to the left of $(1,w_1)$ is a bubble, and every cell to the right is empty.
Therefore, $(1,w_1)$ is the first cell to the right of all the bubbles in the first row, implying $(1,w_1) \in R(\ID(w))$.

Now, we induct and assume that the death ray origin $(i',w_{i'})$ is in $R(\ID(w))$ for all $i' < i$.
There is a death ray origin at $(i,w_i)$, meaning there are no bubbles to the right of $(i,w_i)$.
If all of the cells to the left of $(i,w_i)$ are bubbles we are done, so assume $(i,w_{i'})$ is an empty cell to the left of $(i,w_i)$.
Since the cell is empty, and since row $i$'s death ray origin is to the right, there must be a death ray origin below the empty cell.
This is $(i',w_{i'})$, so $i' < i$ and a row dot was placed there by induction.
Hence, we cannot place a row dot in $(i,w_{i'})$, nor in any empty cell to the left of $(i,w_i)$.
Then we must place the row dot in cell $(i,w_i)$, and we conclude that $R(\ID(w)) = \{(i,w_i) \in \IZ^+\!\times \IZ^+\}$.
Symmetric reasoning proves that $C(\ID(w)) = \{(i,w_i) \in \IZ^+\!\times \IZ^+\}$.
\end{proof}

We now define vertical and horizontal popping rules for any diagram, which correspond to death rays in the Rothe diagram case.

\begin{definition}
The \textbf{\textit{vertical popping}} rule states that bubbles are not allowed to be placed above row dots, and the \textbf{\textit{horizontal popping}} rule states that bubbles are not allowed to be placed to the right of column dots.
\end{definition}

\begin{example}
In Figure \ref{fig:main_va}, we see that $\ID(152869347)$ and $D_2$ satisfy the vertical popping rule, but $D_1$ does not.
Similarly, the horizontal popping rule is satisfied by the same diagrams.
In particular, $D_2$ is an example of a diagram that satisfies both the dot rule and the popping rules, but is not a Rothe diagram, implying we will need more conditions to characterize Rothe diagrams.
\end{example}

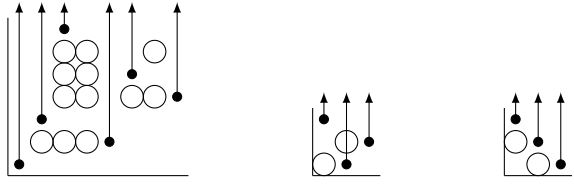
\begin{figure}
\begin{center}
\begin{tikzpicture}[scale=0.3]
\def\rows{8} %number of rows
\def\cols{7} %number of columns
  \draw[-](0,0) -- (\rows,0); %axes
  \draw[-](0,0) -- (0,\cols); 
  \foreach \x/\y in {1/1,2/3,3/7,5/2,6/5,8/4} {%dot positions
    \draw[fill=black] (\x-0.5,\y-0.5) circle(0.2); 
    %vertical lines
    \draw[-latex](\x-0.5,\y-0.5) -- (\x-0.5,7.7); }
  \foreach \x/\y in {2/2,3/2,4/2,3/4,4/4,6/4,7/4,3/5,4/5,3/6,4/6,7/6} %positions
    \draw (\x-0.5,\y-0.5) circle(0.5);
\end{tikzpicture}
\hspace{1.4cm}
\begin{tikzpicture}[scale=0.3]
\def\rows{3} %number of rows
\def\cols{3} %number of columns
  \draw[-](0,0) -- (\rows,0); %axes
  \draw[-](0,0) -- (0,\cols); 
  \foreach \x/\y in {1/3,3/2,2/1} {%dot positions
    \draw[fill=black] (\x-0.5,\y-0.5) circle(0.2); 
    %vertical lines
    \draw[-latex](\x-0.5,\y-0.5) -- (\x-0.5,3.7); }
  \foreach \x/\y in {1/1,2/2} %positions
    \draw (\x-0.5,\y-0.5) circle(0.5);
\end{tikzpicture}
\hspace{1.4cm}
\begin{tikzpicture}[scale=0.3]
\def\rows{3} %number of rows
\def\cols{3} %number of columns
  \draw[-](0,0) -- (\rows,0); %axes
  \draw[-](0,0) -- (0,\cols); 
  \foreach \x/\y in {1/3,2/2,3/1} {%dot positions
    \draw[fill=black] (\x-0.5,\y-0.5) circle(0.2); 
    %vertical lines
    \draw[-latex](\x-0.5,\y-0.5) -- (\x-0.5,3.7); }
  \foreach \x/\y in {2/1,1/2} %positions
    \draw (\x-0.5,\y-0.5) circle(0.5);
\end{tikzpicture}
  \caption{\label{fig:main_va}Vertical ray placement in diagrams $\ID(152869347)$, $D_1$, and $D_2$.}
\end{center}
\end{figure}

We've just proven that death ray origins are equivalent to row and column dots on Rothe diagrams.
Thus, on a Rothe diagram, no bubbles can exist to the right of its column dots or above its row dots, giving us the following statement.

\begin{corollary}
Rothe diagrams satisfy both the vertical and horizontal popping rules.
\end{corollary}

In fact, the dot rule and popping rules are equivalent.

\begin{proposition}
A diagram satisfies the dot rule if and only if it satisfies both popping rules.
\end{proposition}

\begin{proof}
First, assume that a diagram $D$ satisfies the dot rule.
By definition, there are no bubbles to the right of a row dot, and there are no bubbles above a column dot.
Then, since $R(D) = C(D)$, both popping rules are satisfied.

Conversely, assume $D$ satisfies both popping rules.
We prove that $R(D) = C(D)$ by induction on the rows.
Precisely, for each $i = 1,2,...$, we prove that $(i,c_i) = (r_{c_i},c_i)$.
In row 1, if there is any gap between bubbles, a column dot would be placed in the leftmost gap and violate the horizontal popping rule.
Thus in row 1, say there is a bubble in each column $1,...,k$.
Then, the row dot $(1,c_1) = (1,k+1)$ and the column dot $(r_{k+1},k+1) = (1,k+1)$ coincide.

Now, assume that row dots and column dots coincide for all rows below row $i$.
Let $(i,c_i)$ denote the row dot in row $i$.
We wish to say that $(i,c_i)$ is also the column dot in column $c_i$.
So we need row $i$ to be above all bubbles in column $c_i$, we need to have no column dots in row $i$ to the left of $(i,c_i)$, and we need row $i$ to be the first row to satisfy these two conditions.

By vertical popping, no bubbles exist above $(i,c_i)$, so row $i$ is indeed above all bubbles in column $c_i$.

Consider a cell $(i,j)$ such that $j < c_i$.
If it is a bubble, it cannot be a column dot.
If there is a bubble to the right of it, it cannot be a column dot by horizontal popping.
So assume there is no bubble in $(i,j)$ and no bubbles to it's right.
By the minimality of $c_i$, there must be a row $i' < i$ such that $j = c_{i'}$.
That is there must be a row dot $(i',j)$ beneath $(i,j)$.
By induction, this is also a column dot $(r_j,j)$, whence $(i,j)$ cannot be a column dot.
Thus, we have no column dots in row $i$ to the left of $(i,c_i)$.

Lastly, there are no row dots in column $c_i$ below row $i$, so there are no column dots below $(i,c_i)$ either, by induction.
Hence, row $i$ must be the first row to satisfy the two conditions above, and $(i,c_i)$ is the column dot in column $c_i$.
\end{proof}

\subsection{Numbering Condition}
We now define a numbering condition on diagrams, which is satisfied by Rothe diagrams and which is shown to be different than the dot rule by examples.

\begin{definition}
Consider two procedures for labeling the bubbles in a bubble diagram with numbers.
First, we give the bubbles a \textbf{\textit{horizontal numbering}} where in the $i$th row, we label the bubbles from left to right $i,i+1,i+2,$ and so on.
Second, we give the bubbles a \textbf{\textit{vertical numbering}} where in the $j$th column, we label the bubbles from bottom to top $j,j+1,j+2,$ and so on.
We say a bubble diagram satisfies the \textbf{\textit{numbering condition}} and is an \textbf{\textit{enumerated diagram}} if the horizontal numbering and vertical numbering yield the same labels for each bubble.
\end{definition}

\begin{example}{\label{ex:num}}
In Figure \ref{fig:main_number}, $\ID(152869347)$ and $D_1$ satisfy the numbering condition, but $D_2$ does not since its horizontal numbering is different from its vertical numbering.
Note that $D_1$ is numbered but not dotted, and $D_2$ is dotted but not numbered.
Hence neither condition implies the other.
\end{example}

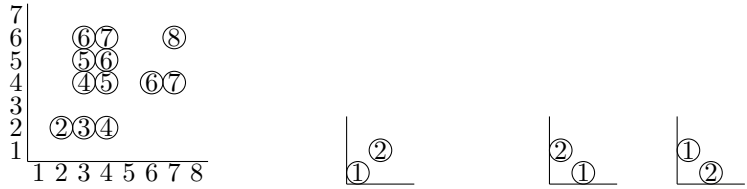
\begin{figure}[ht]
\begin{center}
  \begin{tikzpicture}[scale=0.3]
\def\rows{8} %number of rows
\def\cols{7} %number of columns
  \draw[-](0,0) -- (\rows,0); %axes
  \draw[-](0,0) -- (0,\cols);
  \foreach \y in {1,2,...,\cols} %labels
    \draw (-0.5,\y-0.5) node{\y}; 
  \foreach \x in {1,2,...,\rows}
    \draw (\x-0.5,- 0.5) node{\x};
  \foreach \x/\y/\n in {2/2/2,3/2/3,4/2/4,3/4/4,4/4/5,6/4/6,7/4/7,3/5/5,4/5/6,3/6/6,4/6/7,7/6/8} %positions
    \draw (\x-0.5,\y-0.5) circle(0.5) node{\n};
\end{tikzpicture}
\hspace{1.5cm}
\begin{tikzpicture}[scale=.3]
\def\rows{3} %number of rows
\def\cols{3} %number of columns
  \draw[-](0,0) -- (\rows,0); %axes
  \draw[-](0,0) -- (0,\cols);
  \foreach \x in {1,2} %positions
    \draw (\x-0.5,\x-0.5) circle(0.5) node{\x};
\end{tikzpicture}
\hspace{1.5cm}
\begin{tikzpicture}[scale=.3]
\def\rows{3} %number of rows
\def\cols{3} %number of columns
  \draw[-](0,0) -- (\rows,0); %axes
  \draw[-](0,0) -- (0,\cols);
  \foreach \x/\y in {2/1,1/2} %positions
    \draw (\x-0.5,\y-0.5) circle(0.5) node{\y};
\end{tikzpicture}
\hspace{0.5cm}
\begin{tikzpicture}[scale=.3]
\def\rows{3} %number of rows
\def\cols{3} %number of columns
  \draw[-](0,0) -- (\rows,0); %axes
  \draw[-](0,0) -- (0,\cols);
  \foreach \x/\y in {2/1,1/2} %positions
    \draw (\x-0.5,\y-0.5) circle(0.5) node{\x};
\end{tikzpicture}
  \caption{\label{fig:main_number}Numbering of diagrams $\ID(152869347)$, $D_1$, and $D_2$.}
\end{center}
\end{figure}

\begin{proposition} \label{number_prop}
Rothe diagrams satisfy the numbering condition.
\end{proposition}

\begin{proof}
Consider a bubble $b$ in cell $(i,w_j)$ in a Rothe diagram $\ID(w)$, which is labeled $n$ by horizontal numbering.
This means there are $n-i$ bubbles to the left of $b$ and therefore $w_j-1 - (n-i)$ empty cells to the left of $b$.
As proven in Proposition \ref{pd=>dot}, Rothe diagrams are dotted.
Then, each empty cell to the left of $b$ has a dot below it, since it must be touched by a death ray.
So, there are $w_j+i-n-1$ dots to the left of the $w_j$th column, beneath row $i$.
Thus, there are $w_j+i-n-1$ empty cells beneath $b$ and therefore $i-1 - (w_j+i-n-1) = n - w_j$ bubbles beneath $b$.
Hence, we also give $b$ the label of $n$ by vertical numbering.
\end{proof}

\subsection{Step-out Avoidance}
Since $D_1$ is numbered in Example \ref{ex:num} even though it is not a Rothe diagram, we need an extra condition to characterize Rothe diagrams.
We define that condition now and show that it is satisfied by Rothe diagrams.
Later, we will show it fully characterizes them.

\begin{definition}
For a given diagram that satisfies the numbering condition, define a \textbf{step-out} to be a pair of bubbles numbered $n$ and $n+1$ in cells $(i,w)$ and $(i+k,w+\ell)$ respectively, for positive $k$ and $\ell$. 
In other words, a step-out occurs whenever there exists a bubble numbered $n$ and a bubble strictly north east numbered $n+1$.
We say that an enumerated diagram is \textbf{step-out avoiding} if no pair of bubbles is a step-out in the diagram.
\end{definition}

\begin{example}
In Figure \ref{fig:main_number}, we see that $\ID(152869347)$ is step-out avoiding, but in $D_1$, the pair of bubbles is a step-out.
The definition does not apply to $D_2$ since it does not satisfy the numbering condition.
\end{example}

\begin{proposition} \label{prop:stepout}
Rothe diagrams are step-out avoiding.
\end{proposition}

\begin{proof}
Suppose we have a Rothe diagram $\ID(w)$.
As shown in Proposition \ref{number_prop}, Rothe diagrams are enumerated. 
Take any bubble $b_n$ numbered $n$ in cell $(i,w_j)$ and another bubble $c$ in cell $(i+k,w_j+\ell)$ for positive $k$ and $\ell$.
We claim that the number in $c$ must be greater than $n+1$.

By the numbering condition, there are $n-i$ bubbles in the same row, to the left of $b_n$; call them $b_i,b_{i+1},...,b_{n-1}$.
Let $d$ be one of $\{b_i,...,b_n\}$ and denote its cell by $(i,w_{j'})$.
Since permutation diagrams are dotted, there must be a dot above $d$. 
This dot cannot be in the $i+k$th row by horizontal popping, since $c$ is already in the $i+k$th row, to the right of the $w_{j'}$th column.
Therefore, it must be either above or below the $i+k$th row. 

If the dot is above, then there must be a bubble in cell $(i+k,w_{j'})$ by Lemma \ref{lem:sw_dot}.
If the dot is below row $i+k$, then cell $(i+k,w_{j'})$ must be empty by vertical popping.
Thus, the number of empty cells in the $i+k$th row above any of $\{b_i,...,b_n\}$ equals the number of dots that are above these bubbles and below the $i+k$th row.

This space above bubbles $b_i,...,b_n$ and below the $i+k$th row has $n-i+1$ columns and $k-1$ rows.
Since two dots are never in the same row or column, there are at most $\min\{n-i+1,k-1\}$ dots in this space.
If $\min\{n-i+1,k-1\} = n-i+1$, then $k+i \geq n+2$.
In this case, the number in $c$ is greater than or equal to $n+2$ because horizontal numbering requires the first number in the $k+i$th row to be $k+i$. 
If $\min\{n-i+1,k-1\} = k-1$, then there must be at most $k-1$ empty cells in the $i+k$th row above these bubbles.
Then there are at least $(n-i+1) - (k-1)$ bubbles in the remaining cells above $b_1,...,b_n$.
By horizontal numbering, the number in $B$ is
\begin{equation*}
\begin{aligned}
    i + k + \#\{\text{bubbles to the left of } c\} &\geq i + k + (n-i + 1) - (k-1)\\
    &= n+2.
\end{aligned}
\end{equation*}
\end{proof}

\subsection{Empty Cell Gap Rule}

We now define the final rule and show that Rothe diagrams satisfy it. 
This rule provides a full characterization of Rothe diagrams when paired with the vertical popping rule.
We do not need horizontal popping, and we present the proof of characterization in the next section.

\begin{definition}
In the rule that follows, it will be convenient to include in any diagram a ``column $0$'' which is filled with a bubble in each row.
These bubbles are called \textbf{\textit{basement bubbles}} and are appended to a diagram when considering the rule below. 
However, these bubbles do not represent inversions in a permutation and are not counted towards the total number of bubbles in each row. See Figure \ref{fig:basement} for examples.
\end{definition}

\begin{figure}
\begin{center}
\begin{tikzpicture}[scale=0.3]
\def\rows{8} %number of rows
\def\cols{7} %number of columns
  \foreach \x/\y in {0/1,0/2,0/3,0/4,0/5,0/6,0/7} %basement bubbles
    \filldraw[fill=black!25,draw=black] (\x-0.5,\y-0.5) circle(0.5);
  \draw[-](0,0) -- (\rows,0); %axes
  \draw[-](0,0) -- (0,\cols); 
  \foreach \x/\y/\n in {2/2/2,3/2/3,4/2/4,3/4/4,4/4/5,6/4/6,7/4/7,3/5/5,4/5/6,3/6/6,4/6/7,7/6/8} %bubbles
    \draw (\x-0.5,\y-0.5) circle(0.5);
\end{tikzpicture}
\hspace{1.4cm}
\begin{tikzpicture}[scale=0.3]
\def\rows{3} %number of rows
\def\cols{3} %number of columns
  \foreach \x/\y in {0/1,0/2} %basement bubbles
    \filldraw[fill=black!25,draw=black] (\x-0.5,\y-0.5) circle(0.5);
  \draw[-](0,0) -- (\rows,0); %axes
  \draw[-](0,0) -- (0,\cols); 
  \foreach \x/\y in {1/1,2/2} %bubbles
    \draw (\x-0.5,\y-0.5) circle(0.5);\end{tikzpicture}
\hspace{1.4cm}
\begin{tikzpicture}[scale=0.3]
\def\rows{3} %number of rows
\def\cols{3} %number of columns
  \foreach \x/\y in {0/1,0/2} %basement bubbles
    \filldraw[fill=black!25,draw=black] (\x-0.5,\y-0.5) circle(0.5);
  \draw[-](0,0) -- (\rows,0); %axes
  \draw[-](0,0) -- (0,\cols);
  \foreach \x/\y in {2/1,1/2} %bubbles
    \draw (\x-0.5,\y-0.5) circle(0.5);\end{tikzpicture}
  \caption{\label{fig:basement}The example diagrams with the basement bubbles shaded.}
\end{center}
\end{figure}
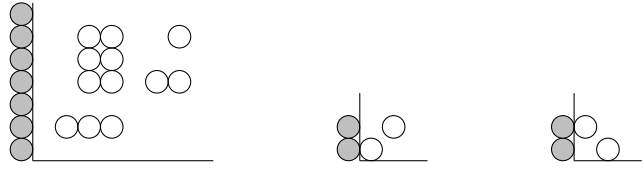

\begin{definition}
The \textbf{\textit{final bubble}} is defined as the last bubble in a row, where all cells afterwards are empty. 
When basement bubbles are taken into account, each row must contain a final bubble. 
See Figure \ref{fig:finals} for examples. 
\end{definition}

These final bubbles are similar to the border cells introduced by Fomin, Greene, Reiner, and Shimozono \cite{fomin1997balanced}. However, here they can be applied to arbitrary diagrams, not just Rothe diagrams, as they do not depend on a corresponding permutation actually existing.

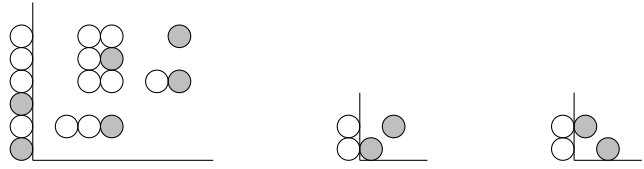
\begin{figure}
\begin{center}
\begin{tikzpicture}[scale=0.3]
\def\rows{8} %number of rows
\def\cols{7} %number of columns
  \foreach \x/\y in {0/2,0/4,0/5,0/6} %positions
    \draw (\x-0.5,\y-0.5) circle(0.5);
  \foreach \x/\y in {0/1,0/3,4/2,7/4,7/6,4/5} %final bubbles
    \filldraw[fill=black!25,draw=black] (\x-0.5,\y-0.5) circle(0.5);
  \draw[-](0,0) -- (\rows,0); %axes
  \draw[-](0,0) -- (0,\cols); 
  \foreach \x/\y in {2/2,3/2,3/4,4/4,6/4,3/5,3/6,4/6} %bubbles
    \draw (\x-0.5,\y-0.5) circle(0.5);
\end{tikzpicture}
\hspace{1.4cm}
\begin{tikzpicture}[scale=0.3]
\def\rows{3} %number of rows
\def\cols{3} %number of columns
  \draw[-](0,0) -- (\rows,0); %axes
  \draw[-](0,0) -- (0,\cols); 
  \foreach \x/\y in {1/1,2/2} %final bubbles
    \filldraw[fill=black!25,draw=black] (\x-0.5,\y-0.5) circle(0.5);
  \foreach \x/\y in {0/1,0/2} %bubbles
    \draw (\x-0.5,\y-0.5) circle(0.5);
\end{tikzpicture}
\hspace{1.4cm}
\begin{tikzpicture}[scale=0.3]
\def\rows{3} %number of rows
\def\cols{3} %number of columns
  \draw[-](0,0) -- (\rows,0); %axes
  \draw[-](0,0) -- (0,\cols); 
  \foreach \x/\y in {2/1,1/2} %final bubbles
    \filldraw[fill=black!25,draw=black] (\x-0.5,\y-0.5) circle(0.5);
  \foreach \x/\y in {0/1,0/2} %bubbles
    \draw (\x-0.5,\y-0.5) circle(0.5);
\end{tikzpicture}
  \caption{\label{fig:finals}The example diagrams with the final bubbles shaded.}
\end{center}
\end{figure}

\begin{definition}
Consider bubbles in cells $(i_0,w_0)$ and $(i_0,w_0+n+1)$ with no bubbles between them, i.e.\ with a horizontal gap of length $n \geq 1$.
Let $\mathcal{B}_{(i_0,w_0+n+1)} \subset \IZ^{\geq 0}\!\times\IZ^+$ be the region below the $i_0$th row and bounded inclusively by the $w_0$th and $(w_0+n)$th columns.
That is,
\[ \mathcal{B}_{(i_0,w_0+n+1)} = \{(i,w) \subset \IZ^{\geq 0}\!\times\IZ^+ \colon i < i_0, w_0 \leq w \leq w_0 + n\}. \]
A diagram satisfies the \textbf{\textit{empty cell gap}} rule if, whenever there is a gap of size $n$, the box $\mathcal{B}_{(i,w+n+1)}$ contains exactly $n$ final bubbles.
\end{definition}

\begin{example}
In Figure \ref{fig:ecg_main}, we see that $\ID(152869347)$ satisfies the empty cell gap rule.
Box $a = \mathcal{B}_{(2, 2)}$ is generated by a gap of length one and contains one final bubble. 
Box $b = \mathcal{B}_{(4, 3)}$, box $c = \mathcal{B}_{(4, 6)}$, and box $d = \mathcal{B}_{(7, 7)}$ contain two, one, and two final bubbles (respectively) and also have corresponding gaps of length two, one, and two (respectively). 
Two of the six empty cell gap boxes are omitted -- $\mathcal{B}_{(5, 3)}$ and $\mathcal{B}_{(6, 3)}$ -- for readability, but it should be noted that these two gaps also satisfy the empty cell gap rule.
Diagram $D_1$ only requires one box $\mathcal{B}_{(2,2)}$ to be checked, and it does satisfy the empty cell gap rule.
Finally, diagram $D_2$ also only creates one box $\mathcal{B}_{(1,2)}$. 
This box has height zero and so cannot contain the single final bubble required by the gap. Therefore this diagram does not fulfill the empty cell gap rule.
\end{example}

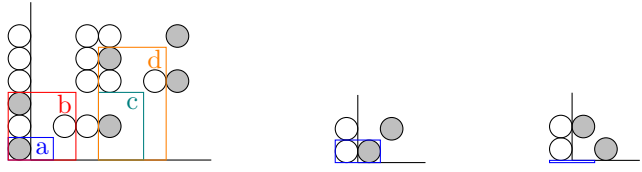
\begin{figure}
\begin{center}
\begin{tikzpicture}[scale=0.3]
\def\rows{8} %number of rows
\def\cols{7} %number of columns
  \foreach \x/\y in {0/2,0/4,0/5,0/6} %positions
    \draw (\x-0.5,\y-0.5) circle(0.5);
  \foreach \x/\y in {0/1,0/3,4/2,7/4,7/6,4/5} %final bubbles
    \filldraw[fill=black!25,draw=black] (\x-0.5,\y-0.5) circle(0.5);
  \draw[-](0,0) -- (\rows,0); %axes
  \draw[-](0,0) -- (0,\cols); 
  \foreach \x/\y in {2/2,3/2,3/4,4/4,6/4,3/5,3/6,4/6} %bubbles
    \draw (\x-0.5,\y-0.5) circle(0.5);
  %rectangles
  \draw[draw=blue] (-1,0) rectangle ++(2,1);
    \draw (0.5,0.5) node{\color{blue} a};
  \draw[draw=red] (-1,0) rectangle ++(3,3);
    \draw (1.5,2.5) node{\color{red} b};
  \draw[draw=teal] (3,0) rectangle ++(2,3);
    \draw (4.5,2.5) node{\color{teal} c};
  \draw[draw=orange] (3,0) rectangle ++(3,5);
    \draw (5.5,4.5) node{\color{orange} d};
\end{tikzpicture}
\hspace{1.4cm}
\begin{tikzpicture}[scale=0.3]
\def\rows{3} %number of rows
\def\cols{3} %number of columns
  \draw[-](0,0) -- (\rows,0); %axes
  \draw[-](0,0) -- (0,\cols); 
  \foreach \x/\y in {1/1,2/2} %final bubbles
    \filldraw[fill=black!25,draw=black] (\x-0.5,\y-0.5) circle(0.5);
  \foreach \x/\y in {0/1,0/2} %bubbles
    \draw (\x-0.5,\y-0.5) circle(0.5);
  \draw[draw=blue] (-1,0) rectangle ++(2,1); %rectangle
\end{tikzpicture}
\hspace{1.4cm}
\begin{tikzpicture}[scale=0.3]
\def\rows{3} %number of rows
\def\cols{3} %number of columns
  \draw[-](0,0) -- (\rows,0); %axes
  \draw[-](0,0) -- (0,\cols); 
  \foreach \x/\y in {2/1,1/2} %final bubbles
    \filldraw[fill=black!25,draw=black] (\x-0.5,\y-0.5) circle(0.5);
  \foreach \x/\y in {0/1,0/2} %bubbles
    \draw (\x-0.5,\y-0.5) circle(0.5);
  \draw[draw=blue] (-1,0) rectangle ++(2,-.1); %rectangle
\end{tikzpicture}
  \caption{\label{fig:ecg_main}The empty cell gap rule applied to $\ID(152869347)$, $D_1$, and $D_2$.}
\end{center}
\end{figure}

\begin{proposition}
Rothe diagrams satisfy the empty cell gap rule.
\end{proposition}

\begin{proof}
For a Rothe diagram, assume a bubble in the cell $(i,w_j)$ and $(i, w_j+n+1)$ with $n$ empty cells in between.
An empty cell in the gap corresponds to a death ray origin below it, since it must be touched by a death ray and is not touched by the horizontal death ray in row $i$.
And there are no death ray origins in the $w_j$th column below the $i$th row by vertical popping.
Then there must be exactly $n$ death ray origins in box $\mathcal{B}_{(i, w_j+n+1)}$.
Then, for each of these death ray origins, there is a bubble in its row in the $w_j$th column by Lemma \ref{lem:sw_dot}.
Hence the $n$ death ray origins correspond to $n$ final bubbles in $\mathcal{B}_{(i, w_j+n+1)}$.

Lastly, we have to show there are no other final bubbles in the box.
If there were, then in some row $i' < i$, we'd have a final bubble in the box with a death ray origin to the right of it outside the box.
This death ray origin must be to the right of the $(w_j+n+1)$th column by vertical popping.
But then Lemma \ref{lem:sw_dot} implies there should be a bubble in cell $(i',w_j+n+1)$, contradicting the finality of the final bubble.
Hence there are exactly $n$ final bubbles in $\mathcal{B}_{(i, w_j+n+1)}$, and the empty cell gap rule is satisfied by Rothe diagrams.
\end{proof}

\section{Characterizations of Rothe Diagrams}

In this section, we combine the above properties to create several full characterizations of Rothe diagrams, completing the proof of Theorem \ref{thm:main}.

\begin{theorem}
A diagram is a Rothe diagram if and only if it satisfies the vertical popping and empty cell gap rules, i.e. (i)$\iff$(ii) in Theorem \ref{thm:main}.
\end{theorem}

\begin{proof}
Arbitrary bubble diagrams that satisfy the cell gap and vertical popping rules are unique, given a number of bubbles for each row. 
In order to prove that these two rules determine a unique bubble diagram (up to rows), we use induction on those rows, starting with the leftmost bubble in each row and moving right. 
We also make use of the basement column in this proof.
\\ \par
\textit{Row $1$:} If there are no bubbles in row one, then the row is all empty cells (besides the basement bubble). 
If there are bubbles in row one, then the row cannot start with any number of empty cells. 
These empty cells would violate the empty cell gap rule (given the basement bubble). 
There are no rows below the first row, so any gap will necessary have an insufficient number of bubbles. 
The same argument applies to the rest of the row. 
Any gaps in the first row of bubbles are prevented by the lack of lower final bubbles. 
Therefore, for a given number of bubbles, row one is unique.
\\ \par
\textit{Row $k+1$:} If there are no bubbles in row $k+1$, then the row is full of empty cells (besides the basement bubble). 
On the other hand, if there are bubbles in row $k+1$, start with the leftmost column, i.e.\ the zeroth column. 
We know by definition of the basement bubbles, it is always filled. 
Take cell $c = (k+1,w_{j}+1)$. 
Either it is preceded by a bubble, or it is preceded by an empty cell. Take each of these scenarios in turn.
\\ \par
First, assume that cell $c$ is directly preceded by a bubble in cell $(k+1,w_{j})$. 
If there is a row dot in cell $(p,w_{j}+1)$ for some $p<k+1$, then cell $c$ must be empty by the vertical popping rule. 
If there does not exist such a row dot, assume for contradiction that cell $c$ is left empty. 
Then, for cell $d_0 = (k+1,w_{j}+2)$ to be a filled with a bubble, there must be exactly one final bubble in box $\mathcal{B}_{d_0}$ (by the empty cell gap rule). 
Column $w_j$ cannot contain this final bubble because if it did, the vertical popping rule would apply to cell $c$. 
So, the final bubble must be in box $\mathcal{B}_{d_0}$'s other column, the $w_{j}+1$ column. 
A final bubble in this column, however, means that cell $d_0$ must be empty by vertical popping. 
Next, take cell $d_{a+1} = (k=1,w_{j}+(a+3)), \; a \geq 0$. 
There must be $a+2$ final bubbles in the $\mathcal{B}_{d_{a+1}}$ box for $d_{a+1}$ to fulfill the empty cell gap rule. 
For each column between $w_{j}$ and $w_{j}+(a+3)$, the final bubble count of box $\mathcal{B}_{d_{a+1}}$ can increase by $n$. 
However, the vertical popping rule would correspondingly add $n$ dots.
Since the first empty cell $d_0$ did not have a final bubble, box $\mathcal{B}_{d_{a+1}}$ will always be at least one final bubble short when the vertical popping rule is fulfilled. 
Therefore, there are bubbles left to place but all remaining locations in the row do not fulfill one or both of the rules. 
So, by contradiction, if the vertical popping rule doesn't apply when a cell $c$ is preceded by a bubble, then $c$ must be a bubble. 
Bubble locations for an arbitrary diagram are uniquely determined when those bubbles are placed directly after other bubbles.
\\ \par
Conversely, take a cell $c = (i,w_j)$ preceded by at least one empty cell. 
Assume without loss of generality that $c$ is preceded by $n$ empty cells.
If $c$ fulfills the empty cell gap rule, then -- by the logic presented above -- a bubble must be placed in cell $c$. 
Otherwise, the empty cell gap rule will only be fulfilled in the case when the vertical popping rule is broken, and the row will not have enough bubbles in it. 
On the other hand, assume cell $c$ does not satisfy the empty cell gap rule. 
Note that the gap must have originally been created because of an application of the vertical popping rule. 
Then, at least one final bubble is in the box $\mathcal{B}_c$ created by this gap. 
If there are $m$ final bubbles in the $i-n$ column, the gap must (by the vertical popping rule) be at least $m$ empty cells long. 
However, if more final bubbles appeared after the $i-n$ column or if $m<n$, then box $\mathcal{B}_c$ will still have too many final bubbles and not enough gaps. 
Since the bubbles on a graph are finite, at some point the number of final bubbles will stop increasing. 
Then, there exists some cell $d$ such that the number of gaps which will equal the number of final bubbles in box $\mathcal{B}_d$. 
Therefore, there is always a unique cell after a gap for a bubble. 
Before this cell, there will be too many empty cells as compared to final bubbles. 
After this cell, there will always be insufficient final bubbles or a contradiction with the vertical popping rule. 
%Note that it is impossible for a graph to have too many required gaps from the bubble popping rule and not enough final bubbles for the empty cell gap rule. 
So, the two rules always create a unique bubble diagram.
\\ \par
We know that for any Lehmer code (equivalently a number of bubbles in each row), there exists a corresponding unique permutation \cite{lehmer1960teaching}. 
As just proven, a unique bubble diagram (up to a number of bubbles in each row) that satisfies the empty cell gap and bubble popping rules also exists. 
Since these two rules are satisfied by permutation diagrams, the diagram generated by these two rules must be identical to the permutation diagram.
\end{proof}

\begin{theorem}
A diagram is a Rothe diagram if and only if it satisfies the dot and numbering conditions, i.e. (i)$\iff$(iii) in Theorem \ref{thm:main}.
\label{thm:DC+NC}
\end{theorem}

\begin{proof}
Consider an arbitrary bubble diagram that satisfies the dot and numbering conditions.
Place the dots on the diagram and let $(1,w_1),(2,w_2),...$ denote their cells.
In fact, since the dots are never in the same column, we are able to refer uniquely to the $w_j$th column and the $(i,w_j)$th cell without confusion.
We claim that a bubble is placed in the cell $(i,w_j)$ exactly when $i < j$ and $w_i > w_j$, which shows that this diagram is the Rothe diagram of the permutation $2 = w_1w_2\cdots.$

Assume there is a bubble in cell $(i,w_j)$.
By the dot condition there must be a dot above and a dot to the right of this cell.
The dot above is in cell $(j,w_j)$ meaning $i < j$, and the dot to the right is in cell $(i,w_i)$ meaning $w_i > w_j$.

Conversely, consider a cell $(i,w_j)$ such that $i < j$ and $w_i > w_j$.
We claim there must be a bubble in any such cell, and we prove it by induction on $i$.
In the base case, the cells to the left of the dot at $(1,w_1)$ are the only ones to satisfy the assumption.
And indeed we must place a bubble in each of these cells.
For if not, then we would place a column dot in a cell to the left of $(1,w_1)$, a contradiction.

Now, let's assume we've shown that there is a bubble in every cell $(i,w_j)$ such that $i < j$ and $w_i > w_j$ for all $i < I$, for some $I > 1$.
Consider a cell $(I,w_j)$ with $I < j$ and $w_I > w_j$, and assume there is no bubble in this cell.
Then, there must be a bubble $b$ in cell $(I,w_{j'})$ for some $w_j < w_{j'} < w_{I}$, for if not then we would place a row dot in a different cell than $(I,w_I)$, a contradiction.
There must be a dot above $b$ by the dot condition.
Then by induction, we see that the number of bubbles below $b$ is equal to the number of dots to the right of the $w_{j'}$th column and below the $I$th row, call this value $d$.
Then by the numbering condition, we would label $b$ with the number $w_{j'} + d$ when counting vertically.

However, we get something different if we count horizontally.
There are $I - 1 - d$ dots to the left of the $w_{j'}$th column and below the $I$th row, none of which are in the $w_j$th column.
That means there are at least $I - d$ empty cells to the left of $b$, since $(I,w_j)$ is empty.
If we count horizontally, then we would label $b$ with the number
$$I + \#\{\text{cells to the left of }A\} - \#\{\text{empty cells to the left of }A\}.$$
This is at most
$$I + (w_{j'} - 1) - (I - d) = w_{j'} + d - 1,$$ 
which is strictly less than our previous label.
This contradicts the numbering condition and our assumption that $(I,w_j)$ contains no bubble must be false.
Hence, this diagram is in fact the Rothe diagram for the permutation $w_1,w_2,...$, showing that any diagram satisfying the dot and numbering conditions is a Rothe diagram.
\end{proof}

\begin{theorem}
A diagram is a Rothe diagram if and only if it satisfies the dot and southwest conditions, i.e. (i)$\iff$(iv) in Theorem \ref{thm:main}.
\end{theorem}

\begin{proof}
Let $D$ be a diagram satisfying the dot and southwest conditions.
Denote its row and column dots by $R(D) = C(D) = \{(1,w_1),(2,w_2),...\}$.
We claim that $D$ is the Rothe diagram of the permutation $w = w_1w_2\cdots$.
By Lemma \ref{lem:sw_dot}, it is sufficient to show that a cell is filled with a bubble in $D$ if and only if the cell is below a dot and to the left of a dot, since these dots correspond to the death ray origins in $\ID(w)$.

If a cell is filled with a bubble in $D$ then there is a dot to the right and a dot above it by the definitions of row dot and column dot respectively.
Conversely, assume a cell $(i,w)$ is to the left of a dot and below a dot.
By the definition of a row dot, the final bubble in row $i$ is followed by columns with dots below the $i$th row and then the row dot $(i,w_i)$.
But since the dot in the $w$th column is above the $i$th row, the final bubble in row $i$ must come in cell $(i,w)$ or to the right of it.
Similarly, using the definition of column dots, there must also be a bubble in $(i,w)$ or above it.
Even if the final bubble in the $i$th row is to the right of $(i,w)$ and the final bubble in the $w$th column is above $(i,w)$, there must be a bubble at $(i,w)$ by the southwest condition.
\end{proof}

\begin{theorem} \label{thm:enum_step_out}
A diagram is a Rothe diagram if and only if it is enumerated and step-out avoiding, i.e. (i)$\iff$(v) in Theorem \ref{thm:main}.
\end{theorem}

\begin{proof}
Given a non negative sequence $a_1,a_2,...$ with finitely many nonzero elements, we claim there exists a unique enumerated step-out avoiding diagram that has $a_i$ bubbles in the $i$th row for all $i$.
There is also a unique permutation diagram that has $a_i$ bubbles in the $i$th row for all $i$.
And since the permutation diagram is enumerated and step-out avoiding, these must be the same diagram.

To show the claim, we construct the enumerated step-out avoiding diagram row by row in the only way possible.
The first row is determined by the numbering condition alone and we fill in the first $a_1$ cells with bubble.
Now assume the first $i-1$ rows have been constructed such that the numbering condition and step-out avoidance are satisfied.
We have $a_i$ bubbles to place in the $i$th row.
If $a_i = 0$ we are done, and if not, then the first bubble must be numbered $i$, so we must place this bubble in a column whose uppermost bubble so far is numbered $i-1$.
If there are multiple options, then we must choose the left-most option, for if not, then the left-most option will form a step-out with the newly placed bubble.
This logic continues for every bubble in this row, and we get a unique placement.
\end{proof}

\section{Characterizations of Free Columns}

We can also apply the previous characterization to a more general situation by considering freely floating columns not yet placed into a fixed diagram, as in the example below.

\begin{example}
Let $\alpha$ be a column with a bubble in row 1 and a bubble in row 2.
Let $\beta$ be a column with bubbles in rows 2, 4, and 5.
Let $\gamma$ be a column with a bubble in row 2.
And let $\delta$ be a column with a bubble in row 5.
In figure \ref{fig:free_cols}, we see that if we place $\alpha, \beta, \gamma$, and $\delta$ in columns 1, 3, 4, and 6, respectively then we get the Rothe diagram $\ID(251463)$.
\end{example}

\begin{figure}
    \centering
\begin{tikzpicture}[scale=0.3, anchor=base, baseline]
\draw[-](0,0) -- (1,0);
\draw (0.5,-.75) node{$\alpha$};
\foreach \y in {1,2}
\draw (0.5,\y-.5) circle(0.5);
\end{tikzpicture}
\hspace{0.3cm}
\begin{tikzpicture}[scale=0.3, anchor=base, baseline]
\draw[-](0,0) -- (1,0);
\draw (0.5,-1) node{$\beta$};
\foreach \y in {2,4,5}
\draw (0.5,\y-.5) circle(0.5);
\end{tikzpicture}
\hspace{0.3cm}
\begin{tikzpicture}[scale=0.3, anchor=base, baseline]
\draw[-](0,0) -- (1,0);
\draw (0.5,-0.75) node{$\gamma$};
\foreach \y in {2}
\draw (0.5,\y-.5) circle(0.5);
\end{tikzpicture}
\hspace{0.3cm}
\begin{tikzpicture}[scale=0.3, anchor=base, baseline]
\draw[-](0,0) -- (1,0);
\draw (0.5,-1) node{$\delta$};
\foreach \y in {5}
\draw (0.5,\y-.5) circle(0.5);
\end{tikzpicture}
\hspace{1cm}
\begin{tikzpicture}[scale=0.3]
\def\rows{6} %number of rows
\def\cols{7} %number of columns
  \draw[-](0,0) -- (\cols,0); %axes
  \draw[-](0,0) -- (0,\rows);
  \foreach \x/\y in {1/1,1/2,3/2,4/2,3/4,3/5,6/5} %positions
    \draw (\x-0.5,\y-0.5) circle(0.5);
\end{tikzpicture}
    \caption{Free columns $\alpha,\beta,\gamma$, and $\delta$, and the Rothe diagram $\ID(251463)$.}
    \label{fig:free_cols}
\end{figure}
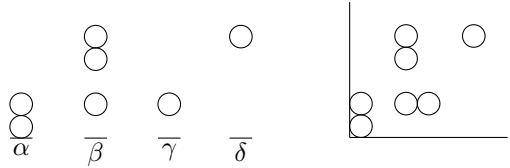

The general question is to figure out when a given set of columns can be placed in cells that yield a Rothe diagram.
Not all given sets of columns form a Rothe diagram; for example, the single column with a bubble in row 1 and a bubble in row 3 cannot be placed in any cell that results in a Rothe diagram.
We find that appropriate definitions of the numbering and step-out avoiding conditions are sufficient to answer this question.

\begin{definition}
A \textbf{collection of free columns} $C$ will be represented by a finite sequence of subsets of $\IZ^+$.
We write $C = \alpha_1,...,\alpha_n$, where $\alpha_j \subset \IZ^+$ for all $j$.
Each $\alpha_j$ represents a column with a bubble in row $i$ for each $i$ in $\alpha_j$.
A \textbf{placement} of $C$ is a sequence of distinct positive integers $w_1 < \cdots < w_n$ resulting in a diagram $D \subset \IZ^+ \times \IZ^+$ defined by $D = \{(i,w_j) \colon i \in \alpha_j, j = 1,...,n\}$, using the French convention.
That is, each free column $\alpha_j$ is placed in the $w_j$th column of diagram $D$.
Note that free columns are not allowed to move past each other horizontally by this definition.
\end{definition}

Thus, our question is, given a collection of free columns $C$, is there a placement of $C$ resulting in a Rothe diagram?
In the previous example we were given $C = \{1,2\}$, $\{2,4,5\}$, $\{2\}$, $\{5\}$, and we saw that the placement $1,3,4,6$ resulted in the Rothe diagram $\ID(251463)$.

\begin{definition}
We give a collection of free columns $C$ a \textbf{horizontal numbering}, where in row $i$, we label the bubbles from left to right $i,i+1,i+2$, and so on. 
We say a collection of free columns $C = \alpha_1,...,\alpha_n$ satisfies the \textbf{numbering condition} if after we give the columns a horizontal numbering, the columns are labeled with unbroken intervals whose starting values strictly increase moving left to right.
That is, the lowest bubble in any column is labeled with some $i$, the next lowest bubble is labeled $i+1$, the next $i+2$, and so on.
And the label of the lowest bubble in $\alpha_1$ is strictly less than the label of the lowest bubble in $\alpha_2$, which is strictly less than the label of the lowest bubble in $\alpha_3$, and so on.
We say such a collection is \textbf{enumerated}.
\end{definition}

\begin{definition}
Once we've given a collection of free columns $\alpha_1,...,\alpha_n$ a horizontal numbering, we say a pair of bubbles is a \textbf{step-out} if it is a step-out in the diagram resulting from any placement of the free columns.
Note that since free columns are not allowed to move past each other in placements, a step-out in one placement is a step-out in any other placement.
A collection is \textbf{step-out avoiding} if it contains no step-outs.
\end{definition}

We can now state and prove the following corrolary of Theorem \ref{thm:main}:

\begin{unnum_corollary}[Corrolary \ref{cor:free_columns}]
For a collection of free columns, there is a unique placement resulting in a Rothe diagram if and only if the collection is enumerated and step-out avoiding.
\end{unnum_corollary}

\begin{proof}
First, assume a collection of free columns has a unique placement resulting in a Rothe diagram $D$.
By Proposition \ref{number_prop}, $D$ is enumerated, which means the $w_j$th column of $D$ is labeled vertically by an unbroken interval starting at the value $w_j$.
Since $w_j$ strictly increases from left to right, the collection of free columns is enumerated as well.
Also, $D$ is step-out avoiding by Proposition \ref{prop:stepout}, implying the collection of free columns is step-out avoiding too.

Conversely, assume a collection of free columns is enumerated and step-out avoiding.
We can construct diagram $D$ by placing each free column in the $j$th column of the diagram for $j$ the label of the lowest bubble in the free column.
This process is well-defined since these $j$-values are distinct and increasing by definition of the numbering condition.
Then the diagram clearly satisfies the numbering and step-out conditions defined originally.
So by Theorem \ref{thm:enum_step_out}, $D$ is a Rothe diagram.

And this is unique because any other diagram $D'$ resulting from a different placement will violate the numbering condition, having a column whose vertical numbering starts at an incorrect value.
\end{proof}

\bibliographystyle{plain}
\bibliography{refs}

\end{document}